\documentclass[11pt]{article}
\usepackage{amsmath,amsthm,amssymb}
\usepackage{url}
\usepackage{geometry}
\newtheorem{theorem}{Theorem}

\newtheorem{lemma}[theorem]{Lemma}
\newtheorem{proposition}[theorem]{Proposition}

\theoremstyle{definition}
\newtheorem{definition}[theorem]{Definition}

\newtheorem{remark}[theorem]{Remark}

\newtheorem{construction}[theorem]{Construction}

\newcommand{\C}{\mathbb{C}}

\newcommand{\B}{\mathcal{B}}

\newcommand{\BIBD}{\mathrm{BIBD}}

\newcommand{\PBD}{\mathrm{PBD}}

\newcommand{\R}{\mathbb{R}}
\newcommand{\maxi}{\mathrm{max}}
\newcommand{\MIP}{\mathrm{MIP}}
\newcommand{\MUB}{\mathrm{MUB}}
\newcommand{\mini}{\mathrm{min}}

\begin{document}

\title{\textbf{\Large{An asymptotic existence result on compressed sensing matrices}}}

\author{
\textsc{Darryn Bryant}
				\thanks{\textit{E-mail: db@maths.uq.edu.au}}\\
\textsc{Padraig \'O Cath\'ain}
				\thanks{\textit{E-mail: p.ocathain@gmail.com}}\\
\textit{\footnotesize{School of Mathematics and Physics,}}\\
\textit{\footnotesize{The University of Queensland, QLD 4072, Australia.}}\\
}

\maketitle
\begin{center}
\begin{abstract}
\noindent
For any rational number $h$ and all sufficiently large $n$ we give a deterministic construction for an $n\times \lfloor hn\rfloor$ compressed sensing matrix with $(\ell_1,t)$-recoverability where $t=O(\sqrt{n})$. Our method uses pairwise balanced designs and complex Hadamard matrices in the construction of $\epsilon$-equiangular frames, which we introduce as a generalisation of equiangular tight frames. The method is general and produces good compressed sensing matrices from any appropriately chosen
pairwise balanced design. The $(\ell_1,t)$-recoverability performance is specified as a simple function of the parameters of the design. To obtain our asymptotic existence result we prove new results on the existence of pairwise balanced designs in which the numbers of blocks of each size are specified.
\end{abstract}

\end{center}

\vspace{0.3cm}



\clearpage
\section{Introduction}\label{CSsection}

Compressed sensing is an approach to data sampling under the hypothesis
that all observations are drawn from a set $X$ of $t$-sparse vectors
in $\R^{N}$, where a vector is said to be {\em $t$-sparse} if it has
at most $t$ non-zero entries, and $t << N$. It has attracted a lot of
attention in the statistics, signal processing, optimization and
computer science literature \cite{CandesCS,DonohoCS} (an extensive
bibliography is maintained at \texttt{http://dsp.rice.edu/cs}).
It is often possible to recover arbitrary $x \in X$ with far fewer
than $N$ linear measurements. Note that while the measurements are linear,
the reconstruction algorithm may be complex and non-linear. We will address
computational and algorithmic aspects of our construction in a subsequent paper.

The problem then is to design an $n\times N$ matrix $\Phi$, with $n$
small relative to $N$, such that the matrix equation $\Phi x = y$
has a unique solution for every $t$-sparse vector $x$. This can be phrased
in terms of $\ell_{0}$-minimization\footnote{Recall that for $p>0$ the $\ell_{p}$
norm of $v$ is $\left(\sum_{1\leq i\leq n} |v_{i}|^{p}\right)^{\frac{1}{p}}$.
The $\ell_{0}$ `norm' of a vector $v$ is given by the limit as $p\rightarrow 0$
of the norm $\ell_{p}$. It counts the number of non-zero entries in $v$.
While it is not a true norm, it is convenient to abuse notation in order
to facilitate a comparison with the $\ell_{1}$ norm.}.
Unfortunately, it is known that $\ell_{0}$-minimization is NP-hard \cite{Natarajan},
so we expect that efficient polynomial-time algorithms for this problem do not exist.
A novelty of compressed sensing lies in the surprising connection between $\ell_{0}$-minimization
and $\ell_{1}$-minimization \cite{CandesRombergTaoStableRecovery}, which
is essentially linear programming, and for which efficient algorithms are available.

\begin{definition}
A matrix $\Phi$ is said to have {\em $(\ell_{1}, t)$-recoverability} if and only if
for any $t$-sparse vector $x$ with $\Phi x = y$, the solution of the linear programming
problem $\Phi \overline{x} = y$ of minimal $\ell_{1}$-norm is equal to $x$.
\end{definition}

\subsection{Mutual incoherence parameters and the Welch bound}

Given a matrix $\Phi$, finding the maximum value of $t$ for which $\Phi$ has
$(\ell_{1}, t)$-recoverability appears to be a computationally difficult problem.
The restricted isometry property (RIP) of Cand\`es, Romberg and Tao is one approach
which has proved to be useful in the analysis of probabilistic constructions \cite{CandesRombergTaoStableRecovery}.
For example, their methods can be used to show that $n\times N$ matrices whose entries
are drawn independently from a Gaussian distribution have the $(\ell_{1}, t)$-recoverability
property with high probability for $t \sim n/\log(N)$, a result which is best possible.
Evaluating the RIP parameters of order $t$ for a matrix $\Phi$ requires knowledge of the
largest and smallest eigenvalues of all principal $t\times t$-submatrices of the $N \times N$
Gram matrix $\Phi^{\top}\Phi$. While random matrix theory can supply this information for
random matrices, applications of RIP have been less successful for deterministic constructions.

Instead, the mutual incoherence parameters (MIP) approach seems to be a standard tool for
establishing $(\ell_{1}, t)$-recoverability results for deterministic compressed sensing
matrices, see \cite{DonohoHuo}. The mutual incoherence parameter $\MIP(\Phi)$ of a matrix
$\Phi$ is defined by
\begin{equation}\nonumber
\MIP(\Phi) = \maxi \frac{|\langle c_{i}, c_{j}\rangle|}{|c_{i}||c_{j}|}
\end{equation}
for distinct columns $c_{i}$ and $c_{j}$ of $\Phi$, where the norm here is $\ell_{2}$.
The following theorem provides the basis for the MIP approach to establishing
$(\ell_{1}, t)$-recoverability.

\begin{theorem}[Proposition 1, \cite{Bourgain}]\label{DonohoHu}
If $\MIP(\Phi) = \mu$, then $\Phi$ has $(\ell_{1}, t)$-recoverability
for all $t < \frac{1}{2\mu} + \frac{1}{2}$.
\end{theorem}

The Welch bound \cite{Welch} states that if $\Phi$ is any $n\times N$ matrix with
columns of unit norm, then
\begin{equation}\label{WelchBoundEquation}
\MIP(\Phi) \geq \mu_{n,N} = \sqrt{\frac{N-n}{(N-1)n}}.
\end{equation}
Combining the Welch bound with Theorem \ref{DonohoHu}, we see that if $(\ell_{1}, t)$-recoverability
can be established for an $n\times N$ matrix $\Phi$ using the MIP approach, then
\begin{equation}\label{restriction}
t \leq \frac{\sqrt{n}}{2}\frac{\sqrt{N-1}}{\sqrt{N-n}} + \frac{1}{2}.
\end{equation}
Provided that $N > \frac{4n-1}{3}$, we have that $t < \sqrt{n}$.
It should be emphasised that a matrix $\Phi$ may support $(\ell_{1}, t)$-recovery for
values of $t$ greater than the bound in (\ref{restriction}), but the MIP framework is
too coarse a tool to establish this. This shortcoming of the MIP approach is known as
the \textit{square-root bottleneck}, and overcoming it requires a fundamentally different
approach to estimating $(\ell_{1}, t)$-recoverability.

For restricted ranges of parameters, there are many known constructions for compressed
sensing matrices close to the square-root bottleneck. Indeed, the first deterministic
construction for compressed sensing matrices by de Vore reached this limit \cite{deVore}.
The only progress towards overcoming the square-root bottleneck is a deep paper of Bourgain
et al. \cite{Bourgain} in which $(\ell_{1}, t)$-recoverability for $t \sim n^{\frac{1}{2} + \epsilon}$
is achieved. Instead of attempting to exceed this threshold, in this paper we will show that matrices
reaching the square-root bottleneck are plentiful and easily constructed.

\section{$\epsilon$-equiangular frames}\label{constructions}

Matrices in which $\MIP(\Phi) = \mu_{n,N}$ (see Equation \eqref{WelchBoundEquation})
are called \textit{equiangular tight frames} (ETFs). We give an overview of some of
their properties in this section. We then introduce a generalisation which we call
$\epsilon$-equiangular frames. If $V$ is a finite dimensional inner product space and
$T$ is a finite subset of $V$, then $T$ is a frame if and only if $T$ spans $V$
(the term originates in functional analysis, and needs refinement if $T$ is infinite).

A frame $T$ is \textit{tight} if there exists a constant $\alpha$ such that for any $v \in V$
\[ \sum_{x\in T} |\langle x,v \rangle|^{2} = \alpha | v|_{2}^{2}, \]
and is \textit{equiangular} if there exists some $\mu \in \mathbb{R}$ such that for
all distinct $x_{i}, x_{j} \in T$
\[ |\langle x_{i}, x_{j} \rangle| = \mu. \]
Note that if $\Phi$ is an $n \times N$ ETF, then elementary arguments
show that any pair of columns necessarily has inner product $\mu_{n,N}$,
see \cite{ExistenceETFs} for a survey of ETFs.

It is known that if $T$ is an ETF in $V$, then $|T| \leq \binom{\dim(V)}{2}$ if
$V$ is real, or $|T| \leq \dim(V)^{2}$ if $V$ is complex. In the case that all
entries of $T$ are contained in some proper subfield of $\C$ (e.g. the rationals or a
cyclotomic field), number theoretic constraints can be used to rule out the existence of
certain large ETFs \cite{ExistenceETFs}. The theory of equiangular tight frames has
applications in quantum computing and functional analysis, however we are interested
in its application to compressed sensing. By Theorem \ref{DonohoHu}, an $n \times N$
ETF has $(\ell_{1},  t)$-recoverability for all
$t \leq \frac{\sqrt{n}}{2}\frac{\sqrt{N-1}}{\sqrt{N-n}} + \frac{1}{2}$,
and so reaches the square-root bottleneck.

Now, for any columns $c_{i}$ and $c_{j}$ of the matrix $\Phi$, let $\mu(c_{i}, c_{j}) = \frac{|\langle c_{i}, c_{j}\rangle|}{|c_{i}||c_{j}|}$
denote the normalised inner product of the vectors $c_i$ and $c_j$ (unless specified otherwise, all measurements are with the $\ell_{2}$-norm).

\begin{definition}\label{epsilonequiangular}
Let $\Phi$ be an $n\times N$ frame and let $\mu_{n,N}$ be the Welch bound.
We say that $\Phi$ is $\epsilon$-equiangular if
\[(1-\epsilon)\mu_{n,N} \leq \mu(c_{i}, c_{j}) \leq (1+\epsilon)\mu_{n,N}\]
for any two distinct columns $c_{i}$ and $c_{j}$ of $\Phi$.
\end{definition}

Our interest in $\epsilon$-equiangular frames stems from the following straightforward result.

\begin{proposition}\label{epscomp}
If $\Phi$ is an $\epsilon$-equiangular frame, then $\Phi$ has the $(\ell_{1}, t)$-recovery property for all $t \leq \frac{\sqrt{n}}{2(1+\epsilon)}$.
\end{proposition}

\begin{proof}
By definition, $\MIP(\Phi) \leq (1+ \epsilon)\mu_{n,N}$.
By Theorem \ref{DonohoHu}, $\Phi$ has
$(\ell_{1}, t)$-recoverability for all
\begin{equation}\label{mipbound}
t \leq \frac{\sqrt{(N-1)n}}{2(1+\epsilon)\sqrt{N-n}} + \frac{1}{2} =
\frac{1}{2(1+\epsilon)} \sqrt{n + \frac{n(n-1)}{N-n}} + \frac{1}{2}.
\end{equation}
So the result holds by observing that $\frac{1}{2(1+\epsilon)}
\sqrt{n}$ is less than the right side of \eqref{mipbound}.
\end{proof}

\begin{remark}
We observe that $2n \leq N$ implies $\frac{n(n-1)}{N-n} \leq n$,
which means that we cannot establish $(\ell_{1},t)$-recoverability
for values of $t$ larger than $\frac{1}{\sqrt{2}(1+\epsilon)} \sqrt{n}$
within the MIP framework. Thus in all cases of interest, the results that
we obtain are best possible to within a (small) multiplicative constant.
\end{remark}

Now we give a construction of $\epsilon$-equiangular frames.
(We give an alternate construction in Section \ref{altconst} which gives
$\epsilon$-equiangular frames for $\epsilon < 1$ and generalises a result
of Fickus, Mixon and Tremain \cite{Mixon}.) We begin by recalling the definition
of a pairwise balanced design. Our use of terminology is standard, and consistent
with \cite{BJL}, for example.

\begin{definition}
If $V$ is a set of $v$ {\em points} and $\B$ is a collection of subsets of $V$,
called {\em blocks}, such that each pair of points occurs together in exactly
$\lambda$ blocks for some fixed positive integer $\lambda$, then $(V,\B)$ is a
{\em pairwise balanced design}. If each block in $\B$ has cardinality in $K$,
then the notation $\PBD(v,K,\lambda)$ is used. For each point $x\in V$, the
{\em replication number} $r_x$ of $x$ is defined by $r_x=|\{B\in\B:x\in B\}|$.
\end{definition}

\begin{construction}\label{con0}
If $(V,\B)$ is a $\PBD(v,K,1)$, then let $n = |\B|$ and $N = \sum_{x \in V} r_{x}$
and define $\Phi$ to be the $n\times N$ frame constructed as follows.

\begin{itemize}
    \item Let $A$ be the transpose of the incidence matrix of $(V,\B)$: rows of $A$
    are indexed by blocks, columns of $A$ by points, and the entry in row $B$ and
    column $x$ is $1$ if $x\in B$ and $0$ otherwise.
    \item For each $x \in V$ let $H_{x}$ be a (possibly complex)
    Hadamard matrix of order $r_{x}$ (see e.g. Section 2.8 of \cite{deLauneyFlannery}).
    \item For each $x \in V$, column $x$ of $A$ determines $r_{x}$
    columns of $\Phi$: each zero in column $x$ is replaced with the $1 \times r_{x}$
    row vector $(0,0,\ldots,0)$, and each $1$ in column $x$ is replaced with a distinct row of $\frac{1}{\sqrt{r_{x}}}H_{x}$.
\end{itemize}
\end{construction}

By Theorem II.2.6 of \cite{BJL}, the rows of $\Phi$ span $\mathbb{R}^{n}$, and so
$\Phi$ is a frame. The substitution of one matrix into another in Construction
\ref{con0} is similar to the column replacement techniques considered in \cite{StrengtheningHash}.
Note that column $x$ of $A$ has precisely $r_{x}$ 1s, which is the number of rows in $H_{x}$.
A standard counting argument implies that $N = \sum_{x \in V} r_{x} = \sum_{B \in \B} |B|$,
so the number of rows in $\Phi$ is the number of blocks in $\B$ and the number of columns is
the sum of the sizes of the blocks in $\B$. These parameters are not directly dependent on $v$,
or on the sizes of individual blocks. Furthermore, we observe that there exist complex Hadamard
matrices of order $r$ for each natural number $r$; the character table of a cyclic group of order $r$
will suffice for the purposes of our construction. (See \cite{LiebeckJames} for the definition
and properties of a character table.)

Lemma \ref{normlemma} shows that while the frames given by Construction \ref{con0} are not tight,
a small modification makes them so. In Proposition \ref{mainthm0} we show that Construction \ref{con0}
does indeed produce $\epsilon$-equiangular frames.

\begin{lemma}\label{normlemma}
If $c\in \mathbb{R}$, $c>0$, and $\Phi$ is a frame from Construction \ref{con0},
then any frame $\Phi'$ produced by normalising every row of $\Phi$ to have length $c$ is tight.
\end{lemma}

\begin{proof}
Since all rows have equal length, it suffices to verify that distinct rows of $\Phi'$ are orthogonal.
Consider two rows of $\Phi'$, labelled by blocks $B_{i}$ and $B_{j} \in \B$. Both rows are non-zero
precisely on the columns labelled by the points in $B_{i} \cap B_{j}$. Orthogonality is obvious if
$B_{i} \cap B_{j}$ is empty, so suppose otherwise. Consider the set of $r_{x}$ columns labelled by
some $x \in B_{i} \cap B_{j}$. These contain two entire rows of a Hadamard matrix, say $h_{i}$ and
$h_{j}$. The inner product restricted to these columns is of the form
$\langle \alpha h_{i}, \beta h_{j} \rangle = \alpha\beta \langle h_{i}, h_{j}\rangle = 0$.
Since the inner product on the set of columns labelled by each point is zero, the result follows.
\end{proof}

If $K$ is a set of integers, then we denote the maximum element of $K$ by $K_{\maxi}$
and the minimum element of $K$ by $K_{\mini}$.

\begin{proposition}\label{mainthm0}
If $(V,\B)$ is a $\PBD(v,K, 1)$ with $2 \leq K_{\mini}$ and $K_{\maxi} \leq \sqrt{2}(K_{\mini}-1)$,
then the frame $\Phi$ produced by Construction \ref{con0} is $1$-equiangular.
\end{proposition}

\begin{proof}
Recall that $\Phi$ is $n\times N$ where $n = |\B|$ and $N$ is the sum of the block sizes
(equivalently the sum of the replication numbers of points), and that $\mu_{n,N}$ is the
Welch bound for $\Phi$, see Equation \eqref{WelchBoundEquation}. Also recall that
$\mu(c_{i}, c_{j}) = \frac{|\langle c_{i}, c_{j}\rangle|}{|c_{i}||c_{j}|}$ and that
$\MIP(\Phi) = \max \mu(c_{i}, c_{j})$, where the maximum is taken over all pairs of
distinct columns of $\Phi$.

It suffices to show that $\mu(c_{i}, c_{j}) \leq 2 \mu_{n,N}$ for any pair $c_{i}$ and $c_{j}$
of columns of $\Phi$. Instead of verifying this inequality directly, we show that
$\frac{1}{2}\frac{K_{\maxi}}{v-1} \leq \mu_{n,N} \leq \MIP(c_{i}, c_{j}) \leq \frac{K_{\maxi}}{v-1}$.
This suffices to show $1$-equiangularity. (Later, in Proposition \ref{mainthm1}, we will replace
$\frac{K_{\maxi}}{v-1}$ with another function of the block sizes of $(V,\B)$ and follow a similar
argument to that given here.)

\begin{enumerate}
        \item For each $x \in V$ we have $\frac{v-1}{K_{\maxi}-1} \leq r_{x} \leq \frac{v-1}{K_{\mini}-1}$.
        By counting pairs of incident points, we can bound the number $n$ of blocks in $\B$, as follows:
        \[  \frac{v(v-1)}{K_{\maxi}(K_{\maxi}-1)} \leq n \leq \frac{v(v-1)}{K_{\mini}(K_{\mini}-1)}. \]

        \item We produce upper and lower bounds on $\mu_{n, N}$ in terms of $K_{\maxi}$ and $v$.
        First, using the given \textbf{lower} bound for $n$,
        \[\mu_{n,N} = \sqrt{\frac{N-n}{(N-1)n}}< \sqrt{\frac{1}{n}} < \sqrt{\frac{K_{\maxi}(K_{\maxi}-1)}{v(v-1)}}< \frac{K_{\maxi}}{v-1}. \]
        (We use $2 < K_{\maxi} < v$ to establish the last inequality.)

        Under the hypothesis (which fails only in degenerate situations) that $2n-1 \leq N$, we have that $\frac{1}{\sqrt{2n}} \leq \mu_{n,N}$.
        Using the given \textbf{upper} bound on $n$, the trivial fact that $v > K_{\maxi}$ and the hypothesis that $\frac{K_{\maxi}}{\sqrt{2}} \leq K_{\mini}-1$, we obtain
        \[ \mu_{n,N} \geq \sqrt{\frac{K_{\mini}(K_{\mini}-1)}{2v(v-1)}} \geq  \frac{1}{\sqrt{2}}\frac{K_{\mini}-1}{v-1} \geq \frac{1}{2} \frac{K_{\maxi}}{v-1}. \]

        \item We have established bounds on $\mu_{n,N}$. To complete the argument it suffices to show that for any
        choice of $i$ and $j$, $\mu(c_{i},c_{j})$ is bounded above by $\frac{K_{\maxi}}{v-1}$. Clearly, any two
        columns of $\Phi$ labelled by the same point $x\in V$ are orthogonal, and hence have inner product $0$.
        (Note that this implies that $\Phi$ cannot be $\epsilon$-equiangular for $\epsilon < 1$.)
        So it suffices to consider columns labelled by distinct points.

        \item If $c_{i}$ and $c_{j}$ are columns labelled by distinct points, then there exists a unique
        row of $\Phi$ in which $c_{i}$ and $c_{j}$ are both non-zero. Every entry $\phi_{xy}$ in the
        matrix $\Phi$ satisfies $\sqrt{\frac{K_{\mini}-1}{v-1}} \leq |\phi_{xy}| \leq \sqrt{\frac{K_{\maxi}-1}{v-1}}$,
        so we have $\frac{K_{\mini}-1}{v-1} \leq |\langle\phi_{ik}, \phi_{jk}\rangle| \leq \frac{K_{\maxi}-1}{v-1}$.
        In particular, $ \mu( c_{i}, c_{j}) \leq \MIP(\Phi) \leq \frac{K_{\maxi}}{v-1}$.
        (The lower bound is irrelevant due to the presence of orthogonal vectors.)
        This completes the proof. \qedhere
\end{enumerate}
\end{proof}

In the statement of Proposition \ref{mainthm0} the constants $\sqrt{2}$ and $1$
were chosen to obtain a neat formulation of the theorem. In fact, for the purposes
of showing $\left(\ell_{1}, O(\sqrt{n})\right)$-recoverability, it suffices to show
that $\max \mu(c_{i}, c_{j}) \leq \alpha\mu_{n,N}$ for some constant $\alpha$.
The structure of Proposition \ref{mainthm0} lends itself to easy adaption to other constants.

We demonstrate that designs not meeting the conditions of Proposition \ref{mainthm0}
can be shown to produce $\epsilon$-equiangular frames for some easily computable value
of $\epsilon$. A $\PBD(v,\{3,5\},1)$ with a single block of size $5$ exists for all
$v \equiv 5 \mod 6$ (see Theorem 6.8, \cite{ColbournRosa}). Applying Construction
\ref{con0} it is easily seen that all non-zero inner products of columns are in the
range $\left[ \frac{2}{v-1}, \frac{2}{v-3}\right]$, while the Welch bound is very
closely approximated by $\frac{2}{\sqrt{(v-4)(v+3)}}$. Thus the matrix obtained from
a $\PBD(v,\{3,5\},1)$ via Construction \ref{con0} is $1$-equiangular for all $v\geq 5$.

We summarise the main results of this section as a theorem.

\begin{theorem}\label{PBDCS}
Let $K$ be a set of integers with $2 \leq K_{\mini}$ and $K_{\maxi} \leq \sqrt{2}(K_{\mini}-1)$.
If there exists a $\PBD(v,K,1)$ with $n$ blocks in which the sum of the block sizes is $N$,
then there exists an $n \times N$ compressed sensing matrix with the $(\ell_{1}, t)$-recovery property
for all $t \leq \frac{\sqrt{n}}{4}$.
\end{theorem}

\begin{proof}
Construction \ref{con0} gives a frame $\Phi$ with $n$ rows and $N$ columns,
Proposition \ref{mainthm0} establishes $1$-equiangularity and
Proposition \ref{epscomp} guarantees $(\ell_{1}, t)$-recovery for all $t \leq \frac{\sqrt{n}}{4}$.
\end{proof}

Theorem \ref{PBDCS} demonstrates that pairwise balanced designs offer a
rich supply of compressed sensing matrices with $(\ell_{1},t)$-recovery
properties close to the square-root bottleneck.

\section{Asymptotic existence of compressed sensing matrices}\label{MainResult}

In this section we use results on the existence of PBDs, which we derive from a
result of Caro and Yuster on asymptotic existence of certain graph decompositions,
to construct compressed sensing matrices. We begin by producing some results,
which we believe to be new, on the existence of $\PBD$s in which the number of blocks of each
size is specified. This builds on an existing literature \cite{WilsonIII,ColbournRodl}.

A {\em decomposition} of a graph $G$ is a set $\mathcal {D}=\{H_1,H_1,\ldots,H_n\}$
of subgraphs of $G$ such that $\bigcup_{i=1}^nE(H_i)=E(G)$ and $E(H_i)\cap E(H_j)=\emptyset$
for $1\leq i<j\leq n$. If $\mathcal{F}$ is a family of graphs and
$\mathcal{D}=\{H_1,H_2,\ldots,H_n\}$ is a decomposition of $G$ such that each $H_i$ is
isomorphic to some graph in $\mathcal {F}$, then $\mathcal{D}$ is called
an $\mathcal{F}$-decomposition. It is clear that a $\PBD(v,K,1)$ is equivalent to an
$\mathcal{F}$-decomposition of $K_v$ where $\mathcal{F}=\{K_k:k\in K\}$.
Here $K_v$ denotes the complete graph on $v$ vertices.
It has an edge joining each pair of distinct vertices.

We shall be using a result of Caro and Yuster \cite{CaroYuster} on $\mathcal{F}$-decompositions.
Their result uses a theorem of Gustavsson \cite{Gustavsson} which has not been published in a
refereed journal, but two independent proofs have recently been published on \texttt{arxiv.org}
by Barber, K\"{u}hn, Lo and Osthus \cite{BarberKuhnLoOsthus}, and also by Keevash \cite{Keevash}.
If $H$ is a graph, then $\gcd(H)$ is defined by $\gcd(H)=\gcd(\{\deg(x):x\in V(H)\})$ where
$\deg(x)$ denotes the degree (in $H$) of the vertex $x$. Let $\mathcal{F}=\{H_1,H_2,\ldots,H_s\}$.
A graph $G$ is said to be {\em $\mathcal{F}$-list-decomposable} if for every list
$\alpha_1,\alpha_2,\ldots,\alpha_s$ of integers satisfying $\sum_{i=1}^s\alpha_i|E(H_i)|=|E(G)|$,
there exists an $\mathcal{F}$-decomposition of $G$ in which the number of copies of $H_i$ is $\alpha_i$ for
$i=1,2,\ldots,s$.

\begin{theorem}[\cite{CaroYuster}, Theorem 1.1]\label{CaroYuster}
If $\mathcal{F}$ is any finite family of graphs such that $\gcd(H)=d$ for each $H\in\mathcal{F}$,
then there exists a constant $C_{\mathcal{F}}$, depending only on $\mathcal{F}$, such that
$K_{n}$ is $\mathcal{F}$-list-decomposable for all $n$ satisfying $n> C_{\mathcal{F}}$ and $d \mid n-1$.
\end{theorem}

We require $\mathcal{F}$-decompositions where $\mathcal{F}$ consists of a number of complete graphs.
Theorem \ref{CaroYuster} cannot be applied directly in this case because $\gcd(K_k)\neq \gcd(K_l)$ for
$k\neq l$. Lemma \ref{blocksmeta} provides a way around this issue, though first we require some more notation.

\begin{definition}
Let $\mathcal D$ be an $\{F_1,F_2,\ldots,F_s\}$-decomposition of $G$ and let
$(F_1,F_2,\ldots,F_s)$ be a given ordering of $F_1,F_2,\ldots,F_s$. The {\em type} of
$\mathcal{D}$ is the vector $(\alpha_1,\alpha_2,\ldots,\alpha_s)$ where $\alpha_i$ is
the number of copies of $F_i$ in $\mathcal D$ for $i=1,2,\ldots,s$. We say that a type $(\alpha_1,\alpha_2,\ldots,\alpha_s)$ is {\em $(G,(F_1,F_2,\ldots,F_s))$-feasible} if $\sum_{i=1}^s\alpha_i|E(F_i)|=|E(G)|$. A $\{K_{k_1},K_{k_2},\ldots,K_{k_s}\}$-decomposition
of $K_v$ is a $\PBD(v,K,1)$ with $K=\{k_1,k_2,\ldots,k_s\}$, and in the context of $\PBD$s we shall
write $(v,(k_1,k_2,\ldots,k_s))$-feasible rather than $(K_v,(K_{k_1},K_{k_2},\ldots,K_{k_s}))$-feasible.
When $G$ and $(F_{1}, F_{2}, \ldots, F_{s})$ are clear from context, we may just write {\it feasible}
rather than $(G,(F_1,F_2,\ldots,F_s))$-feasible.
 \end{definition}

\begin{lemma}\label{blocksmeta}
Let $K=\{k_1,k_2,\ldots,k_s\}$ and let $M$ be an $s \times t$ matrix with
non-negative integer entries with rows indexed by $k_1-1,k_2-1,\ldots,k_s-1$.
Further, suppose that for each column $c$ of $M$, the $\gcd$ of the row indices
of the non-zero entries in $c$ is $1$. There exists a constant $C$ such that
if $v>C$ and $(\alpha_1,\alpha_2,\ldots,\alpha_s)$ is $(v,(k_1,k_2,\ldots,k_s))$-feasible,
then there exists a $(v,K,1)$-PBD of type $(\alpha_1,\alpha_2,\ldots,\alpha_s)$ whenever
\[MX=(\alpha_1,\alpha_2,\ldots,\alpha_s)^{\top}\]
has a solution $X$ in non-negative integers.
\end{lemma}

\begin{proof}
For $j \in \{1, 2, \ldots, t\}$ define the graph $F_{j} = \Sigma_{i=1}^{s} m_{ij} K_{k_i}$.
Here, $m_{ij}$ is the entry in row $i$ and column $j$ of the given matrix $M$, and $m_{1j}K_{k_1}+m_{2j}K_{k_2}+\ldots+m_{sj}K_{k_s}$ is the union of vertex disjoint complete
graphs, where the number of copies of $K_{k_i}$ is $m_{ij}$ for $i=1,2,\ldots,s$.
The hypothesis concerning the columns of $M$ ensures that $\gcd(F_{j}) = 1$ for all $j$.
Thus, by Theorem \ref{CaroYuster}, there exists a constant $C$ such that for all $v>C$,
$K_v$ is $\{F_1,F_2,\ldots,F_t\}$-list-decomposable.

By hypothesis, $MX=(\alpha_1,\alpha_2,\ldots,\alpha_s)^{\top}$ has a solution $(x_1,x_2,\ldots,x_t)^\top$.
Since the type $(\alpha_1,\alpha_2,\ldots,\alpha_s)$ is $(v,(k_1,k_2,\ldots,k_s))$-feasible,
it follows that the type $(x_1,x_2,\ldots,x_t)$ is $(K_v,(F_1,F_2,\ldots,F_s))$-feasible.
Hence there exists an $\{F_1,F_2,\ldots,F_t\}$-decomposition of type $(x_1,x_2,\ldots,x_t)$
(because $K_v$ is $\{F_1,F_2,\ldots,F_t\}$-list-decomposable). For $j=1,2,\ldots,t$, $F_j$ can be
decomposed into $m_{1j}$ copies of $K_{k_1}$, $m_{2j}$ copies of $K_{k_2}$, and so on.
The resulting decomposition of $K_v$ corresponds to
$\PBD(v,K,1)$ of type $(\alpha_1,\alpha_2,\ldots,\alpha_s)$.
\end{proof}

In the case that $M$ is invertible, a $\PBD$ of type $(\alpha_1,\alpha_2,\ldots,\alpha_s)$
exists whenever $M^{-1}(\alpha_1,\alpha_2,\ldots,\alpha_s)^\top$ consists of non-negative
integers. If in addition $M$ is unimodular, we need only check non-negativity.
Proposition \ref{designexistence} illustrates the utility of Lemma \ref{blocksmeta}.

\begin{proposition}\label{designexistence}
If $k>3$ is an integer, then there exists a constant $C_{k}$ such that for every $v > C_{k}$,
there exists a $\PBD(v,\{k-1,k,k+1\},1)$ of type $(\alpha_{k-1}, \alpha_{k}, \alpha_{k+1})$
for every integer solution $(\alpha_{k-1}, \alpha_{k}, \alpha_{k+1})$ of the following linear program.
\begin{eqnarray}
\alpha_{k} & \geq & \alpha_{k-1} \label{ex1}\\
\alpha_{k} & \geq & \alpha_{k+1} \label{ex2}\\
\alpha_{k+1} + \alpha_{k-1} & \geq & \alpha_{k} \label{ex3}\\
\alpha_{k-1}\binom{k-1}{2} + \alpha_{k}\binom{k}{2} + \alpha_{k+1}\binom{k+1}{2} & = & \binom{v}{2} \label{ex4}
\end{eqnarray}
\end{proposition}

\begin{proof}
Let $K = \{k-1, k, k+1\}$, and let
\[ M = \left( \begin{array}{rrr} 1 & 0 & 1 \\ 1 & 1 & 1 \\ 0 & 1 & 1 \end{array} \right). \]
Note that $M$ satisfies the requirements of Lemma \ref{blocksmeta}, with constant $C_{k}$.

Since $M$ is invertible the system $MX = (\alpha_{k-1}, \alpha_{k}, \alpha_{k+1})^{\top}$ is equivalent to
\[ \left( \begin{array}{rrrr}
0 &1 &-1 \\
-1& 1 &0 \\
1 &-1 &1 \end{array} \right) \left( \begin{array}{r} \alpha_{k-1} \\ \alpha_{k} \\ \alpha_{k+1} \end{array}\right) =
\left(\begin{array}{r} x_{1} \\ x_{2} \\x_{3} \end{array} \right).
\]
Now, $M$ is unimodular, so $\left( x_{1}, x_{2}, x_{3} \right)^{\top}$ is integral
when $\left( \alpha_{k-1},  \alpha_{k}, \alpha_{k+1} \right)^{\top}$ is. Clearly,
$x_{1} = \alpha_{x} - \alpha_{k+1}$ is positive precisely when inequality \eqref{ex2} is satisfied.
Likewise, Inequalities \eqref{ex1} and \eqref{ex3} correspond to the second and third rows of this
linear system. It follows that for any integer solution
$\left(\alpha_{k-1}, \alpha_{k}, \alpha_{k+1} \right)^{\top}$ of the system of equations \eqref{ex1}-\eqref{ex4},
$X^{\top} =M\left(\alpha_{k-1}, \alpha_{k}, \alpha_{k+1} \right)^{\top}$ consists of non-negative integers.
Hence by Lemma \ref{blocksmeta} there exists a
$\PBD(v,K,1)$ of type $\left(\alpha_{k-1}, \alpha_{k}, \alpha_{k+1} \right)$.
\end{proof}

Now we turn to the construction of compressed sensing matrices.
The following lemma follows from an easy manipulation of binomial coefficients,
but will be used repeatedly, so we record it here.

\begin{lemma}\label{binomId}
The number of pairs of edges covered by the union of $\alpha_{k-1}$ vertex disjoint
copies of $K_{k-1}$, $\alpha_{k}$ vertex disjoint copies of $K_{k}$ and $\alpha_{k+1}$
vertex disjoint copies of $K_{k+1}$ is $F(\alpha_{k-1},\alpha_{k},\alpha_{k+1}) =
\alpha_{k-1}\binom{k-1}{2} + \alpha_{k}\binom{k}{2} + \alpha_{k+1} \binom{k+1}{2}$.
This function obeys the identity $F(\alpha_{k-1}+t,\alpha_{k}-2t,\alpha_{k+1}+t) =
F(\alpha_{k-1},\alpha_{k},\alpha_{k+1}) + t$.
\end{lemma}

\begin{proposition}\label{integercase}
If $k > 3$ is an integer, then there exists a constant $C_{k}$
such that for all $n>C_{k}$, there exists an $n \times kn$ compressed sensing
matrix with $(\ell_{1}, t)$-recoverability for all $t \leq \frac{\sqrt{n}}{4}$.
\end{proposition}

\begin{proof}
By Theorem \ref{PBDCS}, it is sufficient to construct a $\PBD(v,K,1)$ with $n$ blocks such that
the sum of the block sizes is $kn$. We show that such designs exist for all sufficiently large $n$.

Let $K = \{k-1,k,k+1\}$ and suppose that $v$ is sufficiently large that Proposition \ref{designexistence} holds.
Then every solution to Equations \eqref{ex1}-\eqref{ex4} corresponds to a $\PBD(v,K,1)$, $(V,\B)$, of type $(\alpha_{k-1},\alpha_{k},\alpha_{k+1})$.

Set $n = \alpha_{k-1} + \alpha_{k} + \alpha_{k+1} = |\B|$. For the moment, we assume that
$n \equiv 0 \mod 12$ to reduce the amount of notation we need to employ. We discuss the other
congruence classes at the end of the argument. We require that the number of columns be $kn$, that is
\[kn = \alpha_{k-1}(k-1) + \alpha_{k}k + \alpha_{k+1}(k+1).  \]
This is clearly equivalent to the requirement that $\alpha_{k-1} = \alpha_{k+1}$. Note that only $k$
is specified in the statement of the proposition: in addition to the value of
$(\alpha_{k-1}, \alpha_{k}, \alpha_{k+1})$, we are free to choose the value of $v$.
We now have the following simplified system of inequalities for Proposition \ref{designexistence}:
\begin{eqnarray}
\alpha_{k-1} \; \leq \; \alpha_{k} &\leq& 2\alpha_{k-1} \label{InequalitiesInt}\\
2\alpha_{k-1} + \alpha_{k} & = & n \label{SumInt}\\
\alpha_{k-1} (k^{2}-k+1) + \alpha_{k} \frac{k^{2}-k}{2} &=& \binom{v}{2} \label{EdgesInt}
\end{eqnarray}
(We have expanded the binomial coefficients and gathered like terms in Equation \eqref{EdgesInt}.)

We note that the simultaneous solutions to Equations \eqref{InequalitiesInt} and \eqref{SumInt} are all of the form
\begin{equation}
\left(\alpha_{k-1}, \alpha_{k}, \alpha_{k-1} \right) = \left( \frac{n}{4} + \tau, \frac{n}{2}-2\tau, \frac{n}{4}+\tau \right) \label{taurep}
\end{equation}
for some $0 \leq \tau \leq \frac{n}{12}$. It suffices to show that there exists a solution to
\eqref{EdgesInt} among the vectors of form \eqref{taurep}. We demonstrate this via the function $F(\alpha_{k-1},\alpha_{k},\alpha_{k+1})$ of Lemma \ref{binomId}.

As $\tau$ ranges over the interval $\left[0, \frac{n}{12}\right]$,
$F(\alpha_{k-1}, \alpha_{k}, \alpha_{k+1})$ ranges over the interval
$\left[ n\binom{k}{2} + \frac{n}{4},  n\binom{k}{2} + \frac{n}{3} \right]$.
Clearly every integer in this interval has a unique preimage of the form given
in Equation \eqref{taurep}. This interval is of length $\frac{n}{12} \sim O(n)$.
On the other hand, the distance between consecutive triangular numbers of order $n$
(that is numbers of the form $\binom{v}{2}$ for positive integer $v$) is $O(\sqrt{n})$.
We conclude that for sufficiently large $n$ (guaranteed already by our application of
Proposition \ref{designexistence}) this interval contains many numbers of the form
$\binom{v}{2}$. Furthermore, each equation $F(\alpha_{k-1}, \alpha_{k}, \alpha_{k+1}) = \binom{v}{2}$
in this interval corresponds to solution $(\alpha_{k-1},\alpha_{k},\alpha_{k+1})$ of
the linear program of Proposition \ref{designexistence}. By construction, the design
corresponding to this solution has $n$ blocks and average block size $k$, establishing
the required result in the case that $n \equiv 0 \mod 12$.

The general case $n \equiv i \mod 12$ requires the introduction of an error term
$\iota \equiv -n \mod 12$ in Equation \eqref{taurep} which complicates the presented formulae
and reduces the range of $\tau$ slightly, but does not change the general argument or the
conclusion of the theorem. This completes the proof for every integer $k$ and every $n > C_{k}$. \qedhere
\end{proof}

Theorem \ref{mainexistence} below is an extension of Proposition \ref{integercase} to all rational numbers.
While Theorem \ref{mainexistence} subsumes Proposition \ref{integercase}, we feel that the proof of
Proposition \ref{integercase} illustrates the key concepts without the technical complications of
the proof of Theorem \ref{mainexistence}.

\begin{theorem}\label{mainexistence}
If $h > 3$ is a rational number, then there exists a constant $C_{h}$ such that for all $n > C_{h}$
there exists an $n \times \lfloor hn\rfloor$ matrix with $(\ell_{1},t)$-recoverability for all
$t \leq \frac{\sqrt{n}}{4}$.
\end{theorem}

\begin{proof}
The proof here is similar in outline to that of Proposition \ref{integercase}.
Proposition \ref{designexistence} will not suffice in this case, we will need to apply
Lemma \ref{blocksmeta} directly. Our notation here is as in Proposition \ref{integercase}.
Take $k$ to be the integer closest to $h$, so $h = k + \epsilon$ for $|\epsilon| \leq \frac{1}{2}$.
Take $K = \{k-1, k, k+1\}$. If $\epsilon < \frac{1}{n}$, then Proposition \ref{integercase} applies.
We consider the case $\epsilon \in \left[ \frac{1}{n}, \frac{1}{2}\right]$ first. By Theorem \ref{PBDCS},
the existence of a $\PBD$ satisfying the following linear system for all sufficiently large $n$
will establish the theorem.
\begin{eqnarray}
\sum_{i\in K} \alpha_{i} & = & n \label{gen1} \\
\sum_{i\in K} i\alpha_{i} & = & \lfloor (k+\epsilon)n \rfloor = kn + \lfloor\epsilon n\rfloor \label{gen2}
\end{eqnarray}
For convenience, we write $\sigma = \lfloor\epsilon n\rfloor$. Since
$\sum_{i\in K} \alpha_{i}  =  n$, Equation \ref{gen2} is equivalent to
$\alpha_{k+1} - \alpha_{k-1} = \sigma$, where by hypothesis $1 \leq \sigma \leq \frac{n}{2}$.
We have two linear equations in three unknowns, so solutions are
parameterised by a single variable. It is easily verified that one solution is
$(0, n-\sigma, \sigma)$ and that a vector in the nullspace is $(1,-2,1)$.
So clearly every solution is of the form $(\alpha_{k-1}, \alpha_{k}, \alpha_{k+1}) =
(\tau, n-\sigma-2\tau, \sigma+\tau)$. The number of blocks in a design is a
non-negative integer, so we require this of any putative solution. Observe that
there is a linear order (given by the value of $\tau$) on the solutions and that
the extremal elements of this system are $(0, n-\sigma, \sigma)$ and
$(\frac{n-\sigma}{2}, 0, \frac{n+\sigma}{2})$. By hypothesis, $\sigma \leq \frac{n}{2}$,
so we have at least $\frac{n}{4}$ distinct integer solutions.

Now, for the existence of a $\PBD$, we require that
$\sum_{i \in K} \binom{i}{2} \alpha_{i}  =  \binom{v}{2}$ has a solution. As in Proposition
\ref{integercase}, we consider the function $F(\alpha_{k-1}, \alpha_{k}, \alpha_{k+1})$ of
Lemma \ref{binomId}, supported on the set $(\alpha_{k-1}, \alpha_{k}, \alpha_{k+1}) =
(\tau, n-\sigma-2\tau, \sigma+\tau)$ where $0 \leq \tau \leq \frac{n}{4}$. We will show that
for any choice of $\sigma$, there exists a matrix $M$  as in Lemma \ref{blocksmeta} such that
there exists an interval of length at least $\frac{n}{36}$ on which each solution of the equation
$F(\alpha_{k-1}, \alpha_{k},\alpha_{k+1}) = \binom{v}{2}$ corresponds to a $\PBD(v,K,1)$ of type
$(\alpha_{k-1}, \alpha_{k}, \alpha_{k+1})$. Note that, as $\sigma \rightarrow \frac{n}{2}$,
the inequalities of Proposition \ref{designexistence} fail to hold on an interval of length $O(n)$.

First we deal with the case where $0 \leq \sigma \leq \frac{n}{4}$. The three inequalities
$\eqref{ex1}, \eqref{ex2}$ and $\eqref{ex3}$ are equivalent to the requirement
$2\sigma + 3\tau \leq n \leq 2\sigma + 4\tau$. Recalling that $\sigma = \lfloor \epsilon n \rfloor$,
we have $3\tau \leq (1-2\epsilon)n \leq 4\tau$. Solving for $\tau$ we obtain
$\frac{(1-\epsilon)n}{4}\leq \tau \leq \frac{(1-\epsilon)n}{3}$. This is an interval of length
$\frac{(1-\epsilon)n}{12}$, which is of length $O(n)$ for any $0 \leq \epsilon \leq \frac{1}{4}$.

Now we consider $ \frac{n}{4} \leq \tau \leq \frac{n}{2}$. Let
\[ M = \left( \begin{array}{rrr} 1 & 0 & 0 \\ 5 & 1 & 1 \\ 0 & 1 & 2 \end{array} \right). \]

Inverting $M$ as in Proposition \ref{designexistence}, we obtain the inequalities
\[ \alpha_{k-1} \geq 0, \hspace{0.2cm}\textrm{and}\hspace{0.2cm}
10\alpha_{k-1} + 2\alpha_{k+1} \geq 2\alpha_{k} \geq 10\alpha_{k-1} + \alpha_{k+1}. \]
Substituting for $\sigma$ and $\tau$ as given in our parametrisation
$(\alpha_{k-1}, \alpha_{k},\alpha_{k+1}) = (\tau, n-\sigma-2\tau, \sigma+\tau)$ of the solution space,
we find that we require $16\tau + 4\sigma \geq 2n \geq 15\tau + 3\sigma$. Now, recalling that
$\sigma = \lfloor \epsilon n\rfloor$ with $\frac{1}{4} \leq \epsilon \leq \frac{1}{2}$,
we solve for $\tau$:
\[ \frac{n(1-2\epsilon)}{8} \leq \tau \leq \frac{n(2-3\epsilon)}{15}. \]
This is an interval of length $\frac{n(1+6\epsilon)}{120}$, where $\epsilon \geq \frac{1}{4}$.
It follows that this interval is of length at least $\frac{n}{48}\sim O(n)$.

The density of triangular numbers then establishes the existence of many solutions of
$\sum_{i \in K} \binom{i}{2} \alpha_{i}  =  \binom{v}{2}$ in the solution space of Equations
\eqref{gen1} and \eqref{gen2} for any value of $\epsilon \in \left[ 0, \frac{n}{2} \right]$.

So we have shown that for any $\epsilon \in \left[0, \frac{1}{2} \right]$,
for any $k$ and for all sufficiently large $n$, there exists an interval of
length $O(n)$ on which every feasible type $(\alpha_{k-1}, \alpha_{k}, \alpha_{k+1})$
corresponds to a $\PBD(v,K,1)$ of type $(\alpha_{k-1}, \alpha_{k}, \alpha_{k+1})$.
By construction, each such design has $n$ blocks and the sum of the block sizes is
$\lfloor(k+\epsilon)n\rfloor = \lfloor hn\rfloor$. The argument extends to
$-\frac{1}{2} \leq \epsilon \leq 0$ by swapping the roles of $k-1$ and $k+1$ in the preceding argument.
\end{proof}

\section{Generalisations and modifications}

We have introduced methods of some generality for the construction of compressed sensing matrices.
In the interests of clarity and brevity, we have sketched only the basic ideas.
In this section we give a number of extensions.

\subsection{Extending Construction \ref{con0} using MUBs}

Let $\mathcal{M} = \{M_{0}, M_{1}, \ldots, M_{e}\}$ be a set of orthonormal bases of $\C^{n}$
(written as matrices with the basis vectors as columns). We say that $\mathcal{M}$ is a set of
\textit{mutually unbiased bases} (MUBs) if, for any $i \neq j$, all entries of
$M_{i}^{\dagger}M_{j}$ have absolute value $\frac{1}{\sqrt{n}}$. Without loss
of generality, we take $M_{0} = I$, in which case each $M_{i}$ is a complex Hadamard matrix.
We show that a set of MUBs can be used to increase the number of columns in the
matrices given in Construction \ref{con0} without any loss in $(\ell_{1},t)$-recoverability.

Suppose that $(V,\B)$ is a $\PBD$ in which all points have replication number $r$, and let
$\{M_{0} (=I), M_{1}, \ldots, M_{e}\}$ be a set of MUBs of dimension $r$. Denote by
$\Phi_{i}$ the frame constructed from Construction \ref{con0} using $M_{i}$ throughout.
We claim that the frame $\left[ \Phi_{1} | \Phi_{2} | \ldots | \Phi_{e} \right]$ is $1$-equiangular.
To see this, it suffices to consider the inner product of a column from $\Phi_{i}$ with a column from
$\Phi_{j}$. In the case that the columns are labelled by distinct points, the columns share a single
non-zero entry, so the inner product is of absolute magnitude at most $\frac{1}{r}$. In the case that
the columns are labelled by the same point from $V$, we have that the inner product is $\frac{1}{r}$,
by the definition of the MUBs.

Thus we obtain a frame with the same number of rows as in a naive application of Construction \ref{con0},
but with $e$ times as many columns. This construction is particularly effective when $r$ is a prime power,
as in this case a full set of $r+1$ $\MUB$s exists \cite{KlappeneckerMUBS}. If $(V,\B)$ is a $\BIBD(v,k,1)$
(that is, a $\PBD(v,K,1)$ with $K = \{k\}$), then $|\B| = n = \frac{v(v-1)}{k(k-1)}$, and every point has
replication number $\frac{v-1}{k-1}$. A direct application of Construction \ref{con0} yields a
$\frac{v(v-1)}{k(k-1)} \times \frac{v(v-1)}{k-1}$ compressed sensing matrix. Under the assumption that
$r =\frac{v-1}{k-1}$ is a prime power and using a set of MUBs, we obtain a $\frac{v(v-1)}{k(k-1)} \times \frac{v(v-1)^{2}}{(k-1)^{2}}$ matrix. While in the first case we obtain a ratio $1:k$ between rows
and columns, in the second we obtain a ratio $1:\frac{k(v-1)}{k-1} > v$, which is a substantial improvement.

Of course, the restriction that all replication numbers are equal is merely a convenience.
We are free to replace each Hadamard matrix in Construction \ref{con0} with a set of mutually
unbiased Hadamard matrices. Little is known about the existence of MUBs when the dimension is not
of prime power order, so the practical applications of this observation in the general case may be limited.

\subsection{A generalisation of Construction \ref{con0} using packings}

If $V$ is a set of $V$ points and $\B$ is a collection of subsets of $V$, then
$(V, \B)$ is a \textit{packing} if each pair of points occurs together in at most
one block of $\B$. If each block in $\B$ has cardinality in $K$, then we denote such
a packing by $\PBD(v,K,q)$. (If every pair of points is contained in exactly one block
we recover the definition of a $\PBD$.) It is easily verified that Construction \ref{con0}
produces $\epsilon$-equiangular frames when a packing is used in place of a pairwise
balanced design, provided that there are no points with replication numbers that are too small.

\begin{proposition}
If there exists a $\PBD(v,K,1)$ with $2 \leq K_{\mini}$ and $K_{\maxi} \leq \sqrt{2}(K_{\mini}-1)$
in which the smallest replication number is at least $r_{x} \geq \frac{v-1}{\tau(K_{\mini}-1)}$,
then there exists a compressed sensing matrix with $(\ell_{1}, t)$-recoverability for all
$t \leq \frac{\sqrt{n}}{4\tau}$. If the average block size is $k$, then the ratio of rows to columns is $1: k$.
\end{proposition}

We observe that the existence of dense packings is guaranteed by the
R\"odl `nibble' (see Section 4.7 of \cite{AlonSpencer}, for example).
Whilst results depending on Theorem \ref{CaroYuster} are asymptotic in nature, one can apply the
R\"odl nibble to obtain a packing of any complete graph with graphs from
$\mathcal{F}$. As $v\rightarrow\infty$, this packing will tend to an $\mathcal{F}$-decomposition.
Thus we can in principal construct compressed sensing matrices of any size using this
method, though they will approach Welch-optimality only as $n \rightarrow \infty$.

\subsection{An alternative construction for $\epsilon$-equiangular frames}\label{altconst}

\begin{construction}\label{con1}
If $(V,\B)$ is a $\PBD(v,K,1)$, then let $n = |\B|$ and $N = \sum_{x\in V} r_{x}+1$
and define $\Phi$ to be the $n\times N$ frame constructed as follows.

\begin{itemize}
    \item Let $A$ be the transpose of the incidence matrix of $(V,\B)$, defined precisely as in Construction \ref{con0}.
    \item For each $x \in V$ let $H_{x}$ be a (possibly complex) Hadamard matrix of order $r_{x}+1$.
    \item For each $x \in V$, column $x$ of $A$ determines $r_{x}+1$ columns of $\Phi$: each zero in column $x$ is replaced with the $1 \times (r_{x}+1)$
    row vector $(0,0,\ldots,0)$, and each $1$ in column $x$ is replaced with a distinct non-initial row of $\frac{1}{\sqrt{r_{x}+1}}H_{x}$.
\end{itemize}
\end{construction}

Results analogous to those shown for Construction \ref{con0} hold also for Construction \ref{con1}.
The main interest of Construction \ref{con1} is as a source of $\epsilon$-equiangular frames for $\epsilon < 1$.

\begin{proposition}\label{mainthm1}
If $(V,\B)$ is a $\PBD(v,K,1)$, then Construction \ref{con1} produces an $\frac{K_{\maxi}-K_{\mini}}{K_{\mini}-1}$-equiangular frame.
\end{proposition}

\begin{proof}
Using the techniques developed in Proposition \ref{mainthm0}, it can be shown that
\[ \frac{K_{\mini} -1}{v+K_{\mini}-2} \leq \frac{1}{r_{\maxi} +1} \leq |\langle c_{i}, c_{j} \rangle| \leq \frac{k_{\maxi}-1}{v+k_{\maxi}-2} \leq \frac{k_{\maxi}-1}{v+k_{\mini}-2}. \]
We note in particular that the ratio between the upper and lower bounds is $\frac{K_{\maxi}-1}{K_{\mini}-1}$.

It is almost immediate from the statement of the Welch bound that $\mu_{n,N} \leq \frac{K_{\maxi} -1}{v+K_{\mini}-2}$.
We now produce a sharper lower bound for $\mu_{n,N}$ than that given in Proposition \ref{mainthm0}. Observe that
$\mu_{n,N}$ is an increasing function in $N$ and $N \geq n K_{\mini}$. Recall also that $n \leq \frac{v(v-1)}{K_{\mini}(K_{\mini}-1)}$. Then
\begin{eqnarray*}
\mu_{n,N} & \geq & \sqrt{ \frac{ nK_{\mini} - n }{ n(nK_{\mini} - 1) } } \\
          & \geq & \sqrt{ \frac{ K_{\mini} - 1 }{ nK_{\mini} } } \\
          & \geq & \sqrt{ \frac{ K_{\mini} - 1 }{ K_{\mini} \frac{v(v-1)}{K_{\mini}(K_{\mini}-1)} } } \\
          & \geq & \sqrt{ \frac{ (K_{\mini} - 1)^{2} }{ v(v-1)} } \\
          & \geq & \sqrt{ \frac{ (K_{\mini} - 1)^{2} }{ v^{2}} } \\
          & \geq & \frac{ K_{\mini} - 1 }{ v} \\
          & \geq & \frac{K_{\mini} -1}{v+K_{\mini}-2}
\end{eqnarray*}

Using the ratio between the upper and lower bounds on inner products described above and the fact that
$\mu_{n,N}$ lies between the bounds, we obtain the following:
\[ \frac{K_{\mini}-1}{K_{\maxi}-1} \mu_{n,N} \leq  |\langle c_{i}, c_{j} \rangle| \leq \frac{K_{\maxi} -1}{K_{\mini}-1}\mu_{n,N}.\]

Finally, as in the definition of equiangularity, we solve for $\epsilon$, finding it to be the greater of  $\frac{K_{\maxi}-K_{\mini}}{K_{\maxi}-1}$ and $\frac{K_{\maxi} -K_{\mini}}{K_{\mini}-1}$, completing the proof.
\end{proof}

We note that in the special case that $K = \{k\}$, we achieve a $0$-equiangular frame.
Lemma \ref{normlemma} implies that such a frame is tight, and so we obtain an ETF.
This is the main result of Fickus et al., see Theorem 1 of \cite{Mixon}.

\subsection{Adaptations of Theorem \ref{mainexistence}}

In Proposition \ref{integercase} we showed that in the $(n,N)$ plane,
if we choose any ray through the origin (with integer slope $k$),
there exists a constant $C_{k}$ such that all points of the form $(n,kn)$ at
distance at least $C_{k}$ from the origin correspond to compressed sensing
matrices meeting the square-root bottleneck. We generalised this in Theorem
\ref{mainexistence} to rational slopes, and showed that points close to the ray of form
$(n,\lfloor kn\rfloor)$ corresponded to compressed sensing matrices. We give another
result in this section which shows that small perturbations to the underlying $\PBD$
can be used to obtain $n \times (k+ \epsilon)n$ compressed sensing matrices close to
the square-root bottleneck.

\begin{proposition}\label{smallextension}
If $\Phi$ is an $n \times kn$ compressed sensing matrix with $(\ell_{1}, t)$-recoverability
as considered in Proposition \ref{integercase}, then for each
$\epsilon \in \left[\frac{-1}{12k}, \frac{1}{12k} \right]$ there exists
an $n \times \lfloor(k+\epsilon)n\rfloor$ matrix $\Phi_{\epsilon}$ with $(\ell_{1}, t)$-recoverability.
\end{proposition}

\begin{proof}
First note the elementary identity
\[ (2k-1) \binom{k}{2} = k \binom{k-1}{2} + (k-1)\binom{k+1}{2}, \]
which tells us that $2k-1$ blocks of size $k$ cover the same number of
pairs of points as $k$ blocks of size $k-1$ and $k-1$ blocks of size $k+1$.
But observe that the sum of the block sizes on the left is $k(2k-1)$, whereas the
sum of the block sizes on the right is $2k^{2}-k-1$. That is, we can reduce the
number of columns by $1$ by swapping $2k-1$ blocks of size $k$ for $k$ blocks of size
$k-1$ and $k-1$ blocks of size $k+1$. The inverse operation increases the number of blocks by $1$.

Now, suppose that $\Phi$ is constructed with the maximum possible number of blocks of size $k$,
so $\Phi$ is of type $(\alpha_{k-1}, \alpha_{k}, \alpha_{k+1}) = \left( \frac{n}{4}+\eta, \frac{n}{2}-2\eta, \frac{n}{4}+\eta\right)$, where $\eta$ is close to zero. Then we can reduce the number of columns approximately $\frac{n}{6}\times \frac{1}{2k-1} \approx \frac{n}{12k}$ times before we reach our lower bound
$\frac{n}{3}$ on the number of blocks of size $k$. Likewise, beginning with the minimum number
of blocks of size $k$, we can increase the number of columns approximately $\frac{n}{12k}$ times.
\end{proof}

\section*{Acknowledgements}

The authors acknowledge the support of the Australian Research Council via grants DP120100790 and DP120103067.

\bibliographystyle{abbrv}
\flushleft{
\bibliography{C:/Bibliography/NewBiblio}
}

\end{document}